\documentclass[12pt]{article}

\usepackage{tikz,pgf}
\usepackage{ytableau}
\usepackage{amssymb}
\usepackage{amsthm}
\usepackage{float}

\newtheorem{thm}{Theorem}

\newcommand{\pr}{{\mathbb P}}
\usepackage{graphicx}	
\begin{document}

\begin{center}
 {\large
 Maximal Cells in Shifted Staircase Tableaux and a Quarter-Circle Law
 }

 \bigskip
 Chihiro Kubota, Taizo Sadahiro and Yoshika Ueda
\end{center}

\begin{abstract}
This note presents a numerical and theoretical study of the random standard Young tableaux 
of the shifted staircase shape $(2n-1,2n-3,\ldots,3,1)$.
Using the hook–walk algorithm to generate tableaux uniformly at random,
we evaluate the probability that a given cell contains the maximal label.
We show that the normalized position of the maximal cell converges in distribution 
to the quarter–circle law.
This observation yields the limiting distribution of the first generator 
in a random reduced decomposition of the longest element of $B_n$.
While type $A$ random sorting networks are known to follow the semicircle law 
(Angel et al., 2007; Dauvergne, 2022), 
the type $B$ case exhibits a strikingly different quarter–circle behavior.
Figures and animations are provided to illustrate the typical trajectories 
and the signed permutation matrices observed in our simulations.

\end{abstract}


\section{Introduction}
This note investigates the probability distribution
of the maximal cell in random standard Young tableaux
of the shifted staircase shape $(2n-1, 2n-3, \ldots, 3,1)$.
We explicitly evaluate the probability that a given cell
contains the maximal label, and show that
the asymptotic distribution of the location of the maximal cell
is described by the quarter-circle law.
Through the bijection of Haiman~\cite{haiman1992dual}
between such tableaux and the reduced decompositions
of the longest element in the Coxeter group
$B_n = {\mathfrak S}_n \wr_{[n]} \{\pm 1\}$
of signed permutations, this result yields
the limiting distribution of the first letters
(indexing the generators)
in uniformly random reduced decompositions of the longest element of $B_n$.

A {\em cell} of a standard Young diagram of
the shifted shape $(2n-1, 2n-3, \ldots, 3,1)$
is indexed by a pair $(i,j)$ where $1\leq i\leq n$
and $i\leq j\leq 2n-i$.
Each cell in a standard Young tableau carries
a distinct label from $\{1,2,\ldots,n^2\}$.
Consequently, the maximal label necessarily lies
in a cell of the form $(i,2n-i)$.
We denote by ${\mathbb P}={\mathbb P}_n$
the uniform probability measure on the set of
standard Young tableaux of the above shifted staircase shape.
That is, under ${\mathbb P}$, all tableaux are equally likely.

Our main result is the following.

\begin{thm}\label{thm:main}
Let $T$ be a random standard Young tableau
of the shifted shape $(2n-1, 2n-3, \ldots, 3,1)$
chosen uniformly at random, and let
$(n-S, n+S)$ be the cell of $T$
containing the maximal label $n^2$.
Then
\[
\lim_{n\to\infty}
\Pr\!\left(a \leq \frac{S}{n} \leq b\right)
=
\int_a^b
\frac{4}{\pi}\sqrt{1-x^2}\,dx,
\]
for $0 \leq a \leq b \leq 1$.
\end{thm}

In addition to the analytical result above,
we conducted numerical experiments
on random sorting networks of type~$B$.
The remarkable properties of type~$A$ random sorting networks
have been thoroughly investigated
\cite{angel2007random, dauvergne2022archimedean}.
In contrast, we report numerical experiments
for the type~$B$ case, whose outcome
reveals a strikingly different yet equally fascinating pattern.

 \section{The hook length formula}

 We use the hook length formula counting
 the standard Young tableaux of the
 shifted shape.
 Let $\lambda = (\lambda_1, \lambda_2, \ldots, \lambda_n)$
 be a strict partition of an integer $m$, that is,
 $\lambda$ is a sequence of positive integer
 which satisfies
 $\lambda_1 > \lambda_2 > \ldots > \lambda_n$
 and
 $\lambda_1 + \cdots + \lambda_n = m$.
 The Young diagram of the shifted shape $\lambda$
 is the set of cells each of which is indexed by
 $(i,j)\in{\mathbb Z}^2$, satisfying $1\leq i\leq n,~ i \leq j < \lambda_i + i$.
 We identify $\lambda$ and the Young diagram of the shifted shape $\lambda$.
 The {\em shifted hook} of cell $(i,j)$ is, by definition,
 \[
  H_{i,j} =
  \left\{
  (i,j')\in\lambda\,|\, j' \geq j
  \right\}
  \cup
  \left\{
  (i',j)\in\lambda\,|\, i' \geq i
  \right\}
  \cup
  \left\{
  (j+1,j')\in\lambda\,|\, j' \geq j + 1
  \right\}.
 \]
 Then the {\em shifted hooklength} $h_{i,j}$ of cell $(i,j)$
 is $|H_{i,j}|$.
 Sagan \cite{sagan1980selecting} showed the hooklength formula
 for the shifted Tableaux:
 The number of shifted standard Young tableaux of shape
 $\lambda$ is equal to
 \[
  \frac{|\lambda|!}{\displaystyle \prod_{(i,j)\in\lambda}h_{i,j}}.
 \]
 For example, Figure $\ref{fig:hooklength}$ shows
 the case where shape is $\lambda = (4,2,1)$.
 We have $7!/(6\cdot 5\cdot 4\cdot 3\cdot 2\cdot 1\cdot 1)=7$
 standard tableaux of shape $\lambda$.
 \begin{center}
  \begin{figure}[H]
    \begin{center}
      \begin{tikzpicture}
       \draw (0,0) node {
       \begin{ytableau}
	~ &  &  & \\
	\none & ~ &   \\
	\none & \none &  \\
       \end{ytableau}};
       \draw[fill = black, line width = 1.5] (-0.3,0.6) circle (0.05) -- ++(1.3,0);
       \draw[line width = 1.5] (-0.3,0.6) -- ++(0,-1.2) -- ++(0.6,0);
       \draw[fill = black, line width = 1.5] (-0.3,0.6) circle (0.05);
       \draw[fill = black, line width = 1.5] (0.3,0.6) circle (0.05);
       \draw[fill = black, line width = 1.5] (1,0.6) circle (0.05);
       \draw[fill = black, line width = 1.5] (-0.3,0.) circle (0.05);
       \draw[fill = black, line width = 1.5] (0.3,-0.6) circle (0.05);
       
       \begin{scope}[xshift = 4cm]
	\draw (0,0) node{
	\begin{ytableau}
	 6 & 5 & 4 & 1\\
	 \none & 3 & 2\\
	 \none & \none  & 1 \\
	\end{ytableau}
	};
       \end{scope}
      \end{tikzpicture}

     \tiny
    \begin{ytableau}
     1 & 2 & 3 & 4\\
     \none & 5 & 6\\
     \none & \none  & 7 \\
    \end{ytableau}
    \begin{ytableau}
     1 & 2 & 3 & 5\\
     \none & 4 & 6\\
     \none & \none  & 7 \\
    \end{ytableau}
    \begin{ytableau}
     1 & 2 & 3 & 6\\
     \none & 4 & 5\\
     \none & \none  & 7 \\
    \end{ytableau}
    \begin{ytableau}
     1 & 2 & 3 & 7\\
     \none & 4 & 5\\
     \none & \none  & 6 \\
    \end{ytableau}
    \begin{ytableau}
     1 & 2 & 4 & 5\\
     \none & 3 & 6\\
     \none & \none  & 7 \\
    \end{ytableau}
    \begin{ytableau}
     1 & 2 & 4 & 6\\
     \none & 3 & 5\\
     \none & \none  & 7 \\
    \end{ytableau}
    \begin{ytableau}
     1 & 2 & 4 & 7\\
     \none & 3 & 5\\
     \none & \none  & 6 \\
    \end{ytableau}

     \caption{The shifted shape $(4,2,1)$ and the hook $H_{1,2}$
     whose length is $5$
     (top left),
     hooklengths of cells in the Young diagram (top right),
     and all of the Young tableaux of shape $\lambda$}
     \label{fig:hooklength}
    \end{center}
  \end{figure}
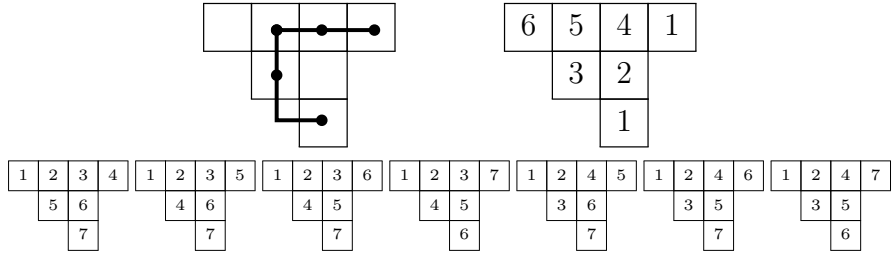
 \end{center}

 \begin{center}
  \begin{figure}[H]
   \begin{center}
   \end{center}
  \end{figure}
 \end{center}

   \begin{thm}\label{thm:exact}
    Let $S$ be as described in Theorem $\ref{thm:main}$. Then
  \[
  \pr\left(S = r\right)
  =
   \frac{(2r+1)(2n-2r-1)}{n^2}\frac{b(r)^2b(n-r-1)b(2r+1)}{b(2r)b(r+n)},
  \]
  where $b(r) = \frac{1}{2^{2r}}{2r\choose r}$.
   \end{thm}
   \begin{proof}
    Let $d(\lambda)$ be the number of the standard
    Young tableaux of shifted shape $\lambda$.
    Let $\lambda$ be the Young diagram of the
    shifted shape $(2n-1, 2n-3, \ldots, 3,1)$
    and let $\lambda'$ be $\lambda\backslash\{(n-r,n+r)\}$.
    The hook length of the cell $(i,j)$ in $\lambda'$
    is denoted by $h'_{i,j}$
    Then the ratio 
    $
    \frac{d(\lambda')}{d(\lambda)}
    $
    is equal to $\pr(S = r)$.
    This ratio is determined by the
    hook lengths of the cells
    whose hook include the cell $(n-r, n+r)$.
    Each such cell is in exactly one of the following sets
    (See Figure $\ref{fig:hooklegnthdiff}$),
   \begin{eqnarray*}
    {\mathcal D}_1 & =  & \left\{(i,n-r-1)\,|\,1 \leq i \leq n-r-1\right\},\\
    {\mathcal D}_2 & =  & \left\{(n-r,j)\,|\, n - r \leq j \leq n-1\right\},    \\
    {\mathcal D}_3 & =  & \left\{(n-r,j)\,|\, n \leq j < n+r \right\},\\
    {\mathcal D}_4 & =  & \left\{(i,n+r)\,|\, 1 \leq i \leq n-r \right\}.
   \end{eqnarray*}
    \begin{center}
     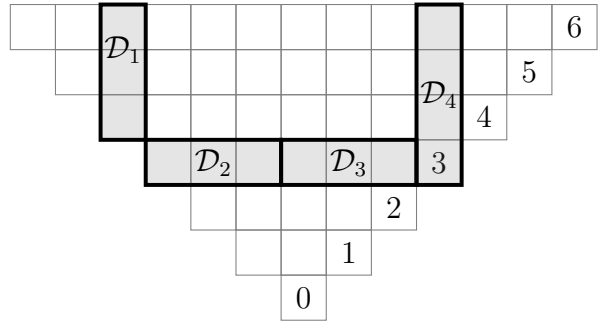
\begin{figure}[H]
      \begin{center}
	\begin{tikzpicture}[scale = 0.6]
	 \foreach \i in {1,...,7}{
	 \pgfmathsetmacro{\n}{15-2*\i}
	 \foreach \j in {1,...,\n}{
	 \draw[color=gray] (\j+\i,-\i) ++(-0.5,-0.5) -- ++(1,0) -- ++(0,1) -- ++(-1,0) -- cycle;
	 }
	 }
	 \foreach \r in {0,...,6}{
	 \draw (8+\r,\r-7) node {\r};
	 }
	 \fill[fill=gray, opacity=0.2] (4,-3) ++ (-0.5,-0.5) -- ++(1,0) -- ++(0,3) -- ++(-1,0) -- cycle;
	 \draw[line width = 1.5] (4,-3) ++ (-0.5,-0.5) -- ++(1,0) -- ++(0,3) -- ++(-1,0) -- cycle;
	 \draw (4,-1.5) node {${\mathcal D}_1$};
	 \fill[fill=gray, opacity=0.2] (8,-4) ++ (-0.5,-0.5) -- ++(0,1) -- ++(-3,0) -- ++(0,-1) -- cycle;
	 \draw[line width = 1.5] (8,-4) ++ (-0.5,-0.5) -- ++(0,1) -- ++(-3,0) -- ++(0,-1) -- cycle;
	 \draw (6,-4) node {${\mathcal D}_2$};
	 \fill[fill=gray, opacity=0.2] (11,-4) ++ (-0.5,-0.5) -- ++(0,1) -- ++(-3,0) -- ++(0,-1) -- cycle;
	 \draw[line width = 1.5] (11,-4) ++ (-0.5,-0.5) -- ++(0,1) -- ++(-3,0) -- ++(0,-1) -- cycle;
	 \draw (9,-4) node {${\mathcal D}_3$};
	 \fill[fill=gray, opacity=0.2] (11,-4) ++ (-0.5,-0.5) -- ++(1,0) -- ++(0,4) -- ++(-1,0) -- cycle;
	 \draw[line width = 1.5] (11,-4) ++ (-0.5,-0.5) -- ++(1,0) -- ++(0,4) -- ++(-1,0) -- cycle;
	 \draw (11,-2.5) node {${\mathcal D}_4$};
	\end{tikzpicture}
       \caption{Cells whose hooklengths in $\lambda$ differ from those in $\lambda'=\lambda - \{(n-r,n+r)\}$,
       where $r=3$}
       \label{fig:hooklegnthdiff}
      \end{center}
     \end{figure}
    \end{center}
    Let ${\mathcal D} = {\mathcal D}_1 \cup {\mathcal D}_2 \cup {\mathcal D}_3 \cup {\mathcal D}_4$ and ${\mathcal D}' = {\mathcal D}\backslash\left\{(n-r,n+r)\right\}$.
    Then, by using the hook length formula,  we have
    \begin{eqnarray*}
    \frac{d(\lambda')}{d(\lambda)}
     & = &
    \frac
    {\displaystyle \prod_{(i,j)\in{\mathcal D}}h_{i,j}}
    {n^2\displaystyle \prod_{(i,j)\in{\mathcal D}'}h'_{i,j}}
    \\
     & = &
      \frac{1}{n^2}
      \frac{\prod_{(i,j)\in{\mathcal D}_1}h_{i,j}}
      {\prod_{(i,j)\in{\mathcal D}_1}h'_{i,j}}{}
      \frac{\prod_{(i,j)\in{\mathcal D}_2}h_{i,j}}
      {\prod_{(i,j)\in{\mathcal D}_1}h'_{i,j}}{}
      \frac{\prod_{(i,j)\in{\mathcal D}_3}h_{i,j}}
      {\prod_{(i,j)\in{\mathcal D}_1}h'_{i,j}}{}
      \frac{\prod_{(i,j)\in{\mathcal D}_4}h_{i,j}}
      {\prod_{(i,j)\in{\mathcal D}_4\backslash\left\{(n-r,n+r)\right\}}h'_{i,j}}{}
    \end{eqnarray*}

    \begin{eqnarray*}
      \frac{\prod_{(i,j)\in{\mathcal D}_1}h_{i,j}}
       {\prod_{(i,j)\in{\mathcal D}_1}h'_{i,j}}{}
       & = &
       \frac{\prod_{k=0}^{n-r-2}\left(4(r+1)+2k\right)}
       {\prod_{k=0}^{n-r-2}\left(4(r+1)+2k-1\right)}\\
     & = & 
      \frac{\prod_{k=r+2}^{n}2(r+k)}
      {\prod_{k=r+2}^{n}\left(2(r+k)-1\right)}\\
     & = & 
      \frac{2^{2(n-r-1)}(4r+2)!\left(\prod_{k=r+2}^{n}(r+k)\right)^2}
      {(2n+2r)!}\\
     & = &
      \frac{1}{2^{4r+2}}{4r+2\choose 2r+1}
      \left(\frac{1}{2^{2n+2r}}{2n+2r\choose n+r}\right)^{-1}\\
     & = &
      \frac{b(2r+1)}{b(n+r)}
    \end{eqnarray*}
    \begin{eqnarray*}
      \frac{\prod_{(i,j)\in{\mathcal D}_2}h_{i,j}}
       {\prod_{(i,j)\in{\mathcal D}_2}h'_{i,j}}{}
       & = &
      \frac{\prod_{k=1}^r\left(2(r+k)-1\right)}{\prod_{k=1}^r2(r+k)}\\
       & = &
	\frac{\prod_{k=1}^r\left(2(r+k)-1\right)\prod_{k=1}^r2(r+k)}
	{\left(\prod_{k=1}^r2(r+k)\right)^2}\\
       & = &
	\frac{(4r)!(r!)^2}
	{2^{2r}(2r)!\left((2r)!\right)^2}
      = 
      \frac{b(2r)}{b(r)}
    \end{eqnarray*}
    \begin{eqnarray*}
      \frac{\prod_{(i,j)\in{\mathcal D}_3}h_{i,j}}
       {\prod_{(i,j)\in{\mathcal D}_3}h'_{i,j}}
       & = & 
      \frac{\prod_{k=1}^r\left(2k+1\right)}{\prod_{k=1}^r2k}
       = 
      \frac{(2r+1)!}{2^{2r}r!r!} = (2r+1)b(r)
    \end{eqnarray*}
    \begin{eqnarray*}
     \frac{\prod_{(i,j)\in{\mathcal D}_4}h_{i,j}}
      {\prod_{(i,j)\in{\mathcal D}_4\backslash\left\{(n-r,n+r)\right\}}h'_{i,j}}{}
      & = &
      \frac{\prod_{k=1}^{n-r-1}\left(2k+1\right)}{\prod_{k=1}^{n-r-1}2k}
      =
      (2(n-r)-1)b(n-r-1)
    \end{eqnarray*}
   \end{proof}

    \begin{proof}[Proof of Theorem\ref{thm:main}]
     We use the Wallis's formula
     \[
      b(n) \sim \frac{1}{\sqrt{\pi{n}}}.
     \]
     Let $r$ be an integer satisfying $a\leq \frac{r}{n} \leq b$.
     Then, by Theorem $\ref{thm:exact}$,
     \begin{eqnarray*}
      \pr(S = r)
       & \sim &
       \frac{(2r+1)(2n-2r-1)}{n^2}
       \frac{\displaystyle \left(\frac{1}{\sqrt{\pi{r}}}\right)^2\frac{1}{\sqrt{\pi(n-r-1)}}\frac{1}{\sqrt{\pi(2r+1)}}}
       {\displaystyle\frac{1}{\sqrt{\pi(2r)}}\frac{1}{\sqrt{\pi(n+r)}}}\\
      & \sim &
       \frac{4}{n\pi}
       \sqrt{\left[1-\left(\frac{r}{n}\right)^2\right]
       \frac{2r(1-r/n)}{(2r+1)(1-(r+1)/n)}
       }\\
      & \sim & 
       \frac{4}{n\pi}
       \sqrt{1-\left(\frac{r}{n}\right)^2}
     \end{eqnarray*}
     Therefore,we have
     \[
     \pr\left(a\leq \frac{S}{n}\leq b\right)
     \sim
     \sum_{a\leq \frac{r}{n}\leq b}
     \frac{4}{n\pi}
     \sqrt{1-\left(\frac{r}{n}\right)^2}
     \rightarrow
     \int_a^b\frac{4}{\pi}\sqrt{1-x^2}\,dx
     ~~~
     (n\rightarrow\infty).
     \]

   \end{proof}

 \section{Numerical experiments on random sorting network of type $B$}

\subsection{Random sorting network and their sampling method}

A {\em signed permutation} of $n$ letters is a sequence
\[
(c_1,c_2,\ldots,c_n),\quad 
c_i\in\{\pm1,\pm2,\ldots,\pm n\},\ |c_i|\neq|c_j|\mbox{ for }i\neq j.
\]
We may think of it as a deck of $n$ labeled cards,
each carrying a sign $+1$ or $-1$ (indicating the direction up or down).
Let $\tau_i$ ($1\le i<n$) denote the adjacent transposition
interchanging $c_i$ and $c_{i+1}$,
and let $\sigma_0$ denote the sign-change operation on the first card.
The group generated by $\{\sigma_0,\tau_1,\ldots,\tau_{n-1}\}$ 
is the hyperoctahedral group $B_n$ of signed permutations.

The {\em longest element} $w_0$ of $B_n$ flips all signs:
\[
w_0(i)=-i \quad (1\le i\le n).
\]
A {\em reduced word} of $w_0$ is a minimal sequence 
$(s_1,s_2,\ldots,s_{N})$ with $s_i\in\{0,1,\ldots,n-1\}$
such that $w_0=\sigma_{s_1}\sigma_{s_2}\cdots\sigma_{s_N}$,
where $\sigma_0=\sigma_0$ and $\sigma_i=\tau_i$ for $i\ge1$.
The number of factors $N$ equals the length $\ell(w_0)=n^2$.

Each reduced word may be viewed as a {\em sorting network}:
starting from the identity permutation, we successively apply 
the adjacent operations $\sigma_{s_i}$, 
each of which swaps or flips neighboring entries.
Plotting the positions of the labels $1,\ldots,n$ against the step index 
gives a collection of $n$ continuous curves,
called the {\em trajectories} of the network.
In type~$B$, these trajectories are constrained by reflection at height~$0$,
corresponding to the action of the generator $\sigma_0$.



Haiman~\cite{haiman1992dual} constructed a bijection between 
standard Young tableaux of the shifted staircase shape 
$(2n-1,2n-3,\ldots,3,1)$
and reduced words of the longest element $w_0$ of the hyperoctahedral group $B_n$.
This bijection enables us to translate probabilistic questions 
on random reduced words into equivalent questions 
on random shifted tableaux.
In particular, 
to generate uniform random reduced words for the type~B longest element, 
we first sample a uniform standard Young tableau (SYT) of 
shifted staircase shape $(2n-1,2n-3,\dots,3,1)$ using 
\emph{Sagan's hook walk algorithm}~\cite{sagan1980selecting}.

\subsection{Frequency of each letter in a random reduced word}
   \begin{figure}[H]
    \begin{center}
     \includegraphics[width=6cm]{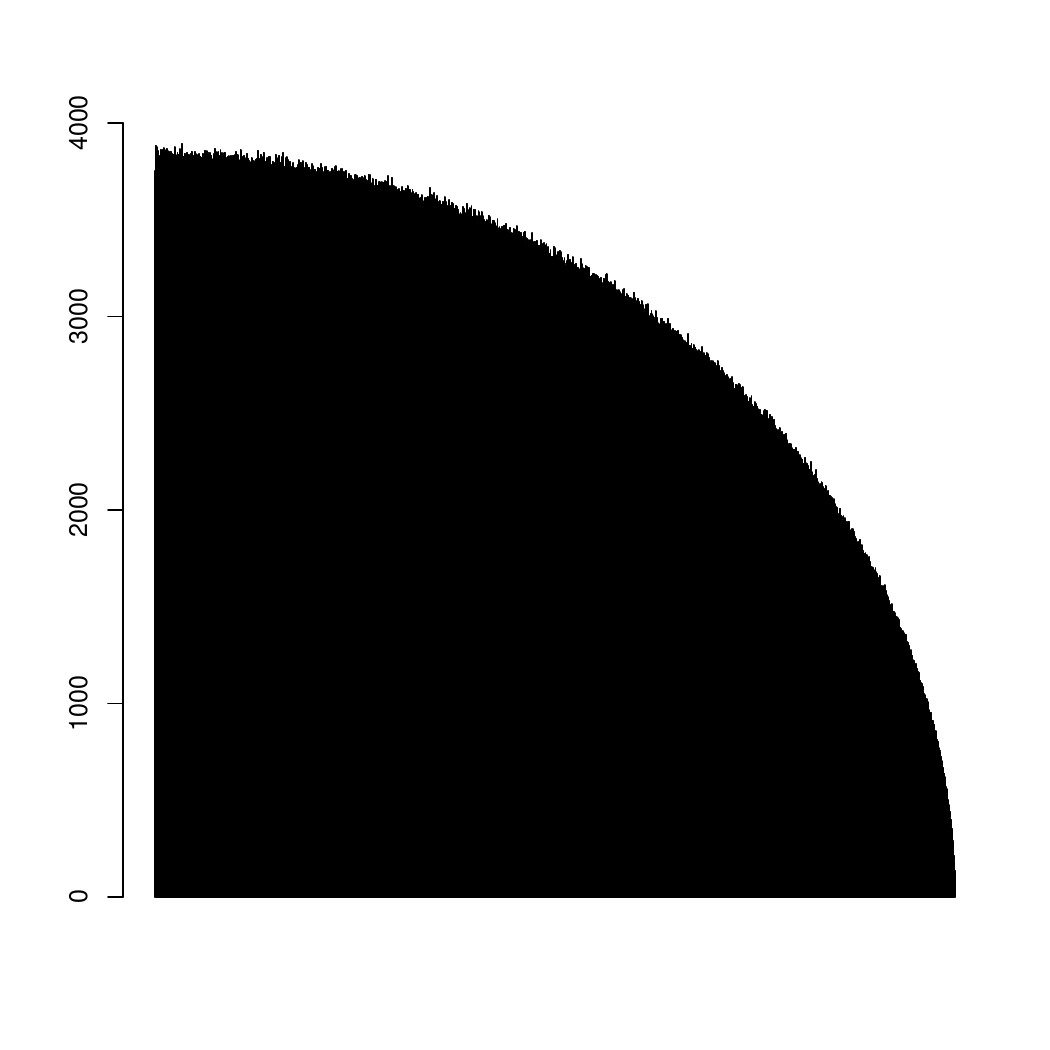}
    \end{center}
    \caption{The frequency of the generating elements of
    a reduced expression of the longest element of $B_{3000}$
    chosen uniformly at random.}
    \label{fig:freqletter}
   \end{figure}

In the previous section, we have shown that the probability distribution of the first
letter in a uniformly random reduced word of the longest element in $B_n$
converges to the quarter-circle law.
Figure~\ref{fig:freqletter} shows the empirical frequencies of all
letters in such a random reduced word of $B_{3000}$, which also exhibit
a distribution close to the quarter-circle law.


\subsection{Trajectories}
Figure~\ref{fig:trajectory} illustrates a typical realization 
of a random type~$B$ sorting network.
Each curve traces the position of a particle labeled $i$ as it evolves
through the sequence of adjacent operations $\sigma_{s_1},\ldots,\sigma_{s_N}$.
The trajectories resemble those of type~$A$ networks 
but are reflected at height~$0$,
corresponding to the action of the generator $\sigma_0$.
\begin{figure}[H]
  \centering
  \includegraphics[width=\textwidth]{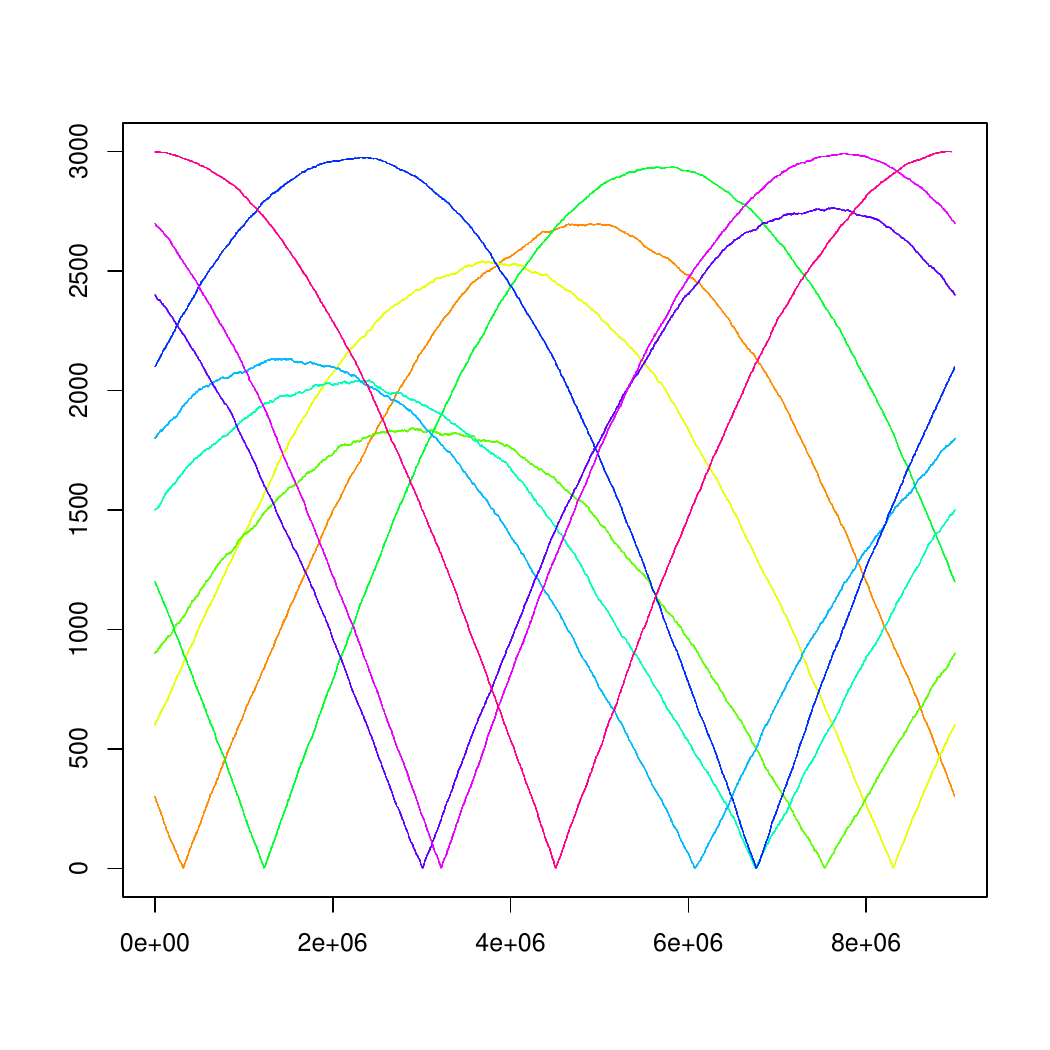}
  \caption{
  Trajectories of particles in a random type-$B$ sorting network ($n=3000$).
  Each curve represents the position of one element as it evolves
  under the sequence of signed adjacent transpositions.
  The trajectories form smooth sinusoidal arcs reflected at height~$0$,
  illustrating the characteristic symmetry of type~$B$.}
  \label{fig:trajectory}
\end{figure}

\subsection{Time evolution of the signed permutation matrix}

\begin{figure}[H]
  \centering
  \includegraphics[width=\textwidth]{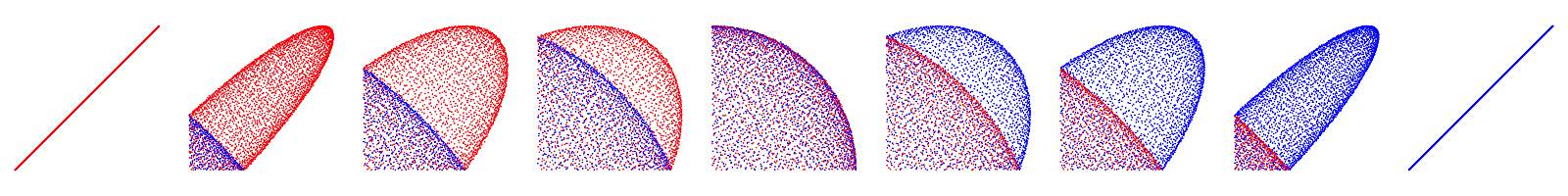}
  \caption{
 Time evolution of the signed permutation matrix
  in a random type-$B$ sorting network ($n=3000$).
  Red and blue dots represent $+1$ and $-1$ entries, respectively.
  The configuration evolves smoothly from the red-dominant (left)
  to the blue-dominant (right) state, exhibiting quarter-circle
  shaped regions that reflect the symmetry of type~$B$.}
  \label{fig:perm_mat_evolution}
\end{figure}

\section*{Conclusion}

We have investigated random sorting networks of type~$B$, 
through their correspondence with shifted standard Young tableaux 
via Haiman's bijection. 
Our analysis and simulations indicate that the normalized position 
of the maximal cell follows the quarter--circle law, 
in contrast to the semicircle law known for type~$A$. 
This provides the first quantitative evidence for a natural type~$B$ analogue 
of the random sorting network model. 
It would be interesting to extend these observations 
to other Coxeter types and to develop a theoretical proof 
of the quarter--circle limit law.


\end{document}